%% file: Current_draft.tex
\documentclass{amsart}

\usepackage{amsmath}
\usepackage{amssymb}
\usepackage{mathrsfs}
\usepackage{stmaryrd}

\usepackage{hyperref}
\usepackage{enumitem}

\usepackage[british]{babel}

\usepackage{comment}
\usepackage{graphicx}
\usepackage[all]{xy}
\usepackage{tikz}
\usepackage{caption}
\usepackage{subcaption}

\usepackage[alphabetic,nobysame]{amsrefs}

\newtheorem{theorem}{Theorem}[section]
\newtheorem{prop}[theorem]{Proposition}
\newtheorem{question}[theorem]{Question}

\theoremstyle{definition}
\newtheorem{remark}[theorem]{Remark}
\newtheorem{example}[theorem]{Example}

\numberwithin{equation}{section}

\input{Commands.tex}

\title[Deformations of log Calabi--Yau pairs can be obstructed]{Deformations of log Calabi--Yau pairs \\ can be obstructed}

\author{Simon Felten}
\address{Department of Mathematics, Columbia University, 2990 Broadway, New York, NY 10027, USA}
\email{sf3127@columbia.edu; felten.math@posteo.net}

\author{Andrea Petracci}
\address{Dipartimento di Matematica, Universit\`a di Bologna, Piazza di Porta San Donato 5, Bologna, 40126, Italy}
\email{a.petracci@unibo.it}

\author{Sharon Robins}
\address{Department of Mathematics, Simon Fraser University, 8888 University Drive, Burnaby BC V5A1S6, Canada}
\email{srobins@sfu.ca}

\begin{document}

\begin{abstract}
We exhibit examples of pairs $(X,D)$ where $X$ is a smooth projective variety  and $D$ is an anticanonical reduced simple normal crossing divisor such that the deformations of $(X,D)$ are obstructed.
These examples are constructed via toric geometry.
\end{abstract}

\maketitle

\section{Introduction}

In the introduction we assume that the ground field, denoted by $\CC$, is  algebraically closed  and has characteristic $0$.

\medskip

A \emph{log Calabi--Yau pair} is a pair $(X,D)$ where $X$ is a smooth projective variety and $D$ is a reduced normal crossing divisor on $X$ such that $K_X + D$ is linearly trivial. 
Most authors, e.g.\ \cite{hacon_xu, kollar_xu, corti_kaloghiros, ask_no_toric_model, mauri}, consider more general definitions and allow singularities on $X$, but we do not do that.

By comparing with the celebrated Bogomolov--Tian--Todorov theorem \cite{bogomolov, tian, todorov} (see also \cite{ran, kawamata, fantechi_manetti, AlgebraicBTT2010, AbstractBTT2017, clm, ffr}), which asserts that deformations of smooth proper varieties over $\CC$ with trivial canonical bundle are unobstructed, one could ask:

\begin{question} \label{question}
Are deformations of log Calabi--Yau pairs $(X,D)$ over $\CC$ unobstructed?
\end{question}

By definition, deformations of a pair $(X,D)$ are (not necessarily locally trivial) deformations of the closed embedding $D \into X$; in particular, the singularities of $D$ are allowed to be smoothed.

The answer to Question~\ref{question} is positive if at least one of the following three additional assumptions is satisfied.
\begin{enumerate}
	
	\item[(i)] $D$ is smooth. This is due to Iacono~\cite{iacono_pairs} and to Sano~\cite[Remark~2.5]{sano} independently
	(see also \cite{Kontsevich_generalized, KKP_Hodge_aspects, LiuRaoWan2019, Wan, felten_petracci}). Note that in this case the deformations of $(X,D)$ are locally trivial and coincide with the log smooth deformations of the log scheme given by $X$ equipped with the divisorial log structure associated to $D$.
	
	\item[(ii)] $X$ is weak Fano, i.e.\ $-K_X$ is big and nef.
	This follows from the unobstructedness of deformations of weak Fano manifolds by Sano~\cite[Theorem~1.1]{sano} and the vanishing of $\rH^1(N_{D/X})$ by Kawamata--Viehweg vanishing.
	
	\item[(iii)] $X$ is a surface. If $X$ is rational, this is due to Friedman~\cite[Proposition~3.5]{friedman}.
	If $X$ is irrational, then $D$ is smooth by a slight generalisation of the proof of \cite[Proposition~1.3]{GHK1} (see also \cite[Proposition~19]{kollar_xu}) and so one concludes thanks to (i).
	\end{enumerate}

\medskip

Examples of log Calabi--Yau pairs are \emph{toric pairs}, i.e.\ pairs $(X,D)$ where $X$ is a smooth projective toric variety and $D$ is the \emph{toric boundary} of $X$, i.e.\ the reduced sum of the torus-invariant prime divisors of $X$. Note that the complement of the toric boundary of $X$ is the big torus of $X$, i.e.\ the open torus-orbit.

If $(X,D)$ is a toric pair, then $X$ is rational, $-K_X$ is big, $D$ has simple normal crossings (i.e.\ the irreducible components of $D$ are smooth), and the pair $(X,D)$ has maximal intersection in the sense that the snc divisor $D$ has $0$-dimensional strata (see \cite{kollar_xu}).

The purpose of this paper is to provide a negative answer to Question~\ref{question} by exhibiting  particular toric pairs $(X,D)$:

\begin{theorem} \label{thm:main}
For every integer $n \geq 3$,
there exists a smooth projective toric $n$-fold $X$ such that the deformations of the pair $(X,D)$, where $D$ denotes the toric boundary of $X$, are obstructed.
\end{theorem}

The idea of the proof is as follows.
With the methods of \cite{ilten_turo} one produces examples of smooth projective toric varieties $X$ which are obstructed.
More precisely, one finds a first-order deformation $\xi$ of $X$ which cannot be extended to the second order.
Let $D$ denote the toric boundary of $X$, and consider the forgetful map
\begin{equation} \label{eq:forgetful}
	\Def{(X,D)} \longrightarrow \Def{X}.
\end{equation}
By using the torus action and the consequent grading on all relevant cohomology groups (see \S\ref{sec:toric-M}), we show that, in some examples, the deformation $\xi$ lies in the image of \eqref{eq:forgetful}.
Obviously, no preimage of $\xi$ can be extended to the second order, therefore we have constructed a first-order deformation of $(X,D)$ which is obstructed.

Actually the varieties $X$ we consider have a description also outside toric geometry: they are products of $\PP^{n-3}$ with the projectivisation of a certain split rank-$2$ vector bundle on the second Hirzebruch surface.

Note that the obstructed first-order deformation $\xi$ of $X$ is not homogeneous with respect to the grading on $\rH^1(X,T_X)$ induced by the Euler sequence; indeed, homogeneous first-order deformations of smooth projective toric varieties are unobstructed \cite{ilten_vollmert}
(see also \cite{mavlyutov, hochenegger_ilten, petracci_manuscripta}).

\medskip

We conclude the introduction with a question in Hodge theory. Assume now that the ground field $\CC$ is the field of complex numbers.
The unobstructedness of deformations of $n$-dimensional smooth projective complex varieties $X$ with trivial canonical bundle  can be proved by using the fact that $\rH^1(T_X) \simeq \rH^1(\Omega_X^{n-1})$ is topological, i.e.\ a Hodge piece of $\rH^n(X,\CC)$ (see the $T^1$-lifting criterion in \cite{ran}).
Similarly, if $(X,D)$ is an $n$-dimensional snc log Calabi--Yau pair, then the \emph{locally trivial} deformations of $(X,D)$ are unobstructed because $\rH^1(T_X(-\log D)) \simeq \rH^1(\Omega_X^{n-1}(\log D))$ is topological, indeed a part of $\rH^n(X \setminus D,\CC)$.
It is an interesting open question, suggested to us by Richard Thomas, to investigate whether there is a topological/Hodge-theoretic explanation of the obstructedness of the (not necessarily locally trivial) deformations of our examples.

\subsection*{Notation and conventions}
In the rest of the paper the ground field, denoted by $\KK$, is an arbitrary field of characteristic different from $2$.
However, \S\ref{sec:defo_pairs} and Proposition~\ref{prop:ses} are valid over a field of arbitrary characteristic.


\subsection*{Acknowledgements}
The first named author thanks the IH\'ES for its hospitality. The second named author wishes to thank Richard Thomas for several fruitful conversations concerning Question~\ref{question},
and is grateful to Enrica Floris and to Anne-Sophie Kaloghiros for helpful discussions about birational geometry.
The third named author would like to thank Nathan Ilten for helpful conversations.
All authors thank Donatella Iacono and Nathan Ilten for useful comments on a preliminary version of this article.

\section{Cup product} \label{sec:cup_product}

Let $X$ be a smooth variety over $\KK$.
Let $\Def{X}$ denote the deformation functor of $X$.
It is well known that $\rH^1(T_X)$ is the tangent space of $\Def{X}$ and 
$\rH^2(T_X)$ is an obstruction space of $\Def{X}$, where $T_X$ is the tangent sheaf of $X$, i.e.\ the sheaf of $\KK$-derivations of $\cO_X$.

By composing the cohomology product $\rH^1(T_X) \otimes_\KK \rH^1(T_X) \to \rH^2(T_X \otimes_\KK T_X)$ with the morphism induced on $\rH^2$ by the Lie bracket $[ \cdot, \cdot ] \colon T_X \otimes_\KK T_X \to T_X$, one gets a symmetric bilinear form
\begin{equation} \label{eq:cup_product}
	b \colon \rH^1(T_X) \otimes_\KK \rH^1(T_X) \longrightarrow \rH^2(T_X)
\end{equation}
which is called the \emph{cup product}.
If one chooses an affine open cover of $X$ and uses \v{C}ech cohomology to describe $\rH^i(T_X)$, then the cup product is given by
\begin{equation*}
b(\xi,\xi') = \{ [\xi_{ij}, \xi_{jk}'] \}_{i,j,k}
\end{equation*}
for $\xi= \{ \xi_{ij} \}_{i,j}$ and $\xi'= \{ \xi_{ij}' \}_{i,j}$.
If one uses alternating \v{C}ech cocycles, then 
\begin{equation*}
b(\xi,\xi') =	 \left\{\frac{ [\xi_{ij}, \xi_{jk}'] - [ \xi_{jk}, \xi_{ij}' ] }{2} \right\}_{i,j,k}.
\end{equation*}
We refer the reader to \cite[\S2]{ilten_turo} for a thorough account.

Let
\[
q \colon \rH^1(T_X) \longrightarrow \rH^2(T_X)
\]
be the quadratic form associated to $b$, i.e.\ $q(\xi) = b(\xi,\xi)$.
In terms of \v{C}ech cocycles we have
\[
q \left( \left\{ \xi_{ij} \right\}_{i,j}  \right) = \left\{ \left[ \xi_{ij}, \xi_{jk}  \right] \right\}_{i,j,k}.
\]
The quadratic form $q$ is very important in deformation theory and is called the \emph{first obstruction}: if $\xi \in \rH^1(T_X)$ is a first-order deformation of $X$ (i.e.\ a deformation of $X$ over $\Spec \KK[t]/(t^2)$), then $\frac{1}{2} q(\xi) \in \rH^2(T_X)$ is the obstruction to lift $\xi$ to a deformation over $\Spec \KK[t] / (t^3)$.

This implies that, up to a multiplicative constant, the quadratic form $q$ coincides with the degree $2$ terms of the equations which define the base of the miniversal deformation of $X$ as a closed subspace of $\rH^1(T_X)$.
In particular, if $q$ is non-zero, then $\Def{X}$ is not smooth, i.e.\ $X$ is obstructed.

\section{Deformations of pairs} \label{sec:defo_pairs}

Let $X$ be a smooth variety over $\KK$, and
let $D$ be an effective (Cartier) divisor on $X$. Let $\Def{(X,D)}$ be the deformation functor of the pair $(X,D)$, i.e.\ of the closed embedding $D \into X$.
There is an obvious natural transformation \eqref{eq:forgetful} which forgets the divisor $D$.

Let $N_{D/X}$ be the normal bundle of $D$ in $X$: this is $\cHom_{\cO_X} (\cO_X(-D), \cO_D)$.
Consider the map
\begin{equation} \label{eq:d}
d \colon T_X \longrightarrow N_{D/X}
\end{equation}
which maps every $\KK$-derivation $\partial \colon \cO_X \to \cO_X$ to the composition $\pi \circ \partial \vert_{\cO_X(-D)}$, where $\partial \vert_{\cO_X(-D)}$ is the restriction of $\partial$ to the ideal $\cO_X(-D)$ of $D$ in $X$ and $\pi \colon \cO_X \onto \cO_D$ is the surjection induced by the closed embedding $D \into X$.
It is easy to show that $d$ is a well-defined homomorphism of $\cO_X$-modules.

Let $A^\bullet$ be the $2$-term cohomological complex given by $d$, i.e.\ $T_X$ is the term in degree $0$, $N_{D/X}$ is the term in degree $1$, all the other terms are zero, and the differential from degree $0$ to degree $1$ is $d$.
There is an obvious exact triangle
\begin{equation*}
	A^\bullet \longrightarrow T_X \overset{d}{\longrightarrow} N_{D/X} \longrightarrow
\end{equation*}
which induces the long exact sequence
\begin{align*}
0 & &\longrightarrow  &	&\HH^0(A^\bullet) & &\longrightarrow & &\rH^0(T_X) & &\longrightarrow & &\rH^0(N_{D/X}) & & & \\
& &\longrightarrow  &	&\HH^1(A^\bullet) & &\longrightarrow & &\rH^1(T_X) & &\longrightarrow & &\rH^1(N_{D/X}) & & & \\
& &\longrightarrow  &	&\HH^2(A^\bullet) & &\longrightarrow & &\rH^2(T_X) & &\longrightarrow & &\rH^2(N_{D/X})&  . & &
\end{align*}
By \cite[Proposition~8]{smith_varley} $\HH^1(A^\bullet)$ is the tangent space of $\Def{(X,D)}$ and the map induced by \eqref{eq:forgetful} on the tangent spaces coincides with the map $\HH^1(A^\bullet) \to \rH^1(T_X)$ in the long exact sequence above.

\begin{remark}
One could prove that $\HH^2(A^\bullet)$ is an obstruction space for $\Def{(X,D)}$, but we will not need this result below.
We just remark that the identification of $\HH^2(A^\bullet)$ as an obstruction space for $\Def{(X,D)}$ allows one to recover the well-known criterion that says that the forgetful map~\eqref{eq:forgetful} is smooth if $\rH^1(N_{D/X}) = 0$; indeed, under the assumption $\rH^1(N_{D/X}) = 0$, the map~\eqref{eq:forgetful} induces a surjection on tangent spaces and an injection on obstruction spaces.
\end{remark}

\begin{remark}
All the discussion above works for pairs $(X,D)$ where $X$ is a smooth variety and $D$ is an effective Cartier divisor on $X$.
Note that $D$ might be non-reduced.
If, in addition, one assumes that $D$ is smooth, then one gets the residue sequence
\begin{equation*}
	0 \longrightarrow \Omega_X \longrightarrow \Omega_X (\log D) \longrightarrow \cO_D \longrightarrow 0,
\end{equation*}
and by dualising it one gets the short exact sequence
\begin{equation*}
	0 \longrightarrow T_X(-\log D) \longrightarrow T_X \longrightarrow N_{D/X} \longrightarrow 0,
\end{equation*}
which shows
that $T_X(-\log D)$ is quasi-isomorphic to $A^\bullet$.
If $D$ is only assumed to be snc, then $A^\bullet$ might not be quasi-isomorphic to $T_X(-\log D)$: indeed, whereas $A^\bullet$ controls all deformations of the pair $(X,D)$, $T_X(-\log D)$ controls only locally trivial deformations of $(X,D)$.
\end{remark}

\section{Toric geometry and $M$-gradings}
\label{sec:toric-M}

Let $N$ be a lattice of rank $n$, let $M$ be its dual, and let $\langle \cdot, \cdot \rangle \colon M \times N \to \ZZ$ be the duality pairing.
Consider the torus $T_N = N \otimes_\ZZ 
\mathbb{G}_\mathrm{m} = \Spec \KK[M]$, where $\mathbb{G}_\mathrm{m} = \Spec \KK[x,x^{-1}]$ is the $1$-dimensional algebraic torus.
If one has a fan $\Sigma$ in $N$, one gets a toric variety equipped with an action of $T_N$.

The set of the $1$-dimensional cones (aka rays) of $\Sigma$ is denoted by $\Sigma(1)$, and the divisor corresponding to $\rho \in \Sigma(1)$ is denoted by $D_{\rho}$.
With slight abuse of notation, the primitive lattice generator of a ray $\rho\in \Sigma(1)$ is denoted again by $\rho$. For more details about toric varieties we refer the reader to \cite{cls}.

If $X$ is a toric variety and $U \subseteq X$ is an \emph{affine toric} open subscheme (i.e.\ the affine toric variety associated to a cone in the fan $\Sigma$ defining $X$), then the torus action on $U$ induces a natural $M$-grading on $\Gamma(U, \cO_X(D))$ for every torus-invariant $\ZZ$-divisor $D$.
If $U' \subseteq U$ is a smaller affine toric subscheme, then the restriction maps preserve the $M$-grading.
In particular, the \v{C}ech complex with respect to the open affine covering given by the maximal cones of the fan $\Sigma$ is naturally $M$-graded; therefore, one has $M$-gradings on the cohomology groups $\rH^i(X, \cO_X(D))$ for every torus-invariant $\ZZ$-divisor $D$.
Note that the isomorphism class of the sheaf $\cO_X(D)$ depends only on the linear equivalence class of $D$, whereas the $M$-gradings on $\rH^i(X, \cO_X(D))$ depend on the divisor $D$.
The homogeneous part of $\rH^i(X, \cO_X(D))$ of degree $u \in M$ is denoted by $\rH^i(X, \cO_X(D))_u$.

Now let $X$ be a smooth toric variety. Using the $M$-grading on $\cO_X(D)$, we induce an $M$-grading on $T_X$ via the (dual) Euler sequence. For $D$ the toric boundary, we get also an $M$-grading on the normal bundle $N_{D/X}$ such that $d\colon T_X \to N_{D/X}$ preserves the gradings. In order to construct the gradings and show that they are preserved by $d$, we employ the following proposition, which holds more generally for every effective divisor $D$ on $X$.

\begin{prop} \label{prop:ses}
Let $X$ be a smooth complete toric variety, let $D$ be the effective divisor on $X$ defined by the zero-locus of a homogeneous polynomial $F$ in the Cox ring of $X$, let $\beta \in \Pic(X)$ be the degree of $F$.

Then there is a commutative diagram of coherent sheaves on $X$ with exact rows
\begin{equation*}
	\xymatrix{
0 \ar[r] & \rN_1(X) \otimes_\ZZ \cO_X \ar[r]^{\! \! \! \! \! \! \! (\beta_\rho x_\rho)_{\rho}} \ar[d]^\beta & \bigoplus_{\rho \in \Sigma(1)} \cO_X(D_\rho) \ar[d]^{\left( \frac{\partial F}{\partial x_\rho} \right)_\rho} \ar[r] & T_X \ar[d]^d \ar[r] & 0 \\
0 	\ar[r] & \cO_X \ar[r]^F & \cO_X(D) \ar[r] & N_{D/X} \ar[r] & 0
}
\end{equation*}
where $x_\rho$ denotes the Cox coordinate associated to the ray $\rho$ and $\beta_\rho \in \Pic(X)$ denotes its degree, the group $\rN_1(X)$ is $\Hom_\ZZ(\Pic(X), \ZZ)$, the top exact sequence is the dual of the Euler sequence \cite[Theorem~8.1.6]{cls}, and the right vertical map is the homomorphism~\eqref{eq:d}.
\end{prop}

By \cite[Equation~(5)]{ilten_turo} the homomorphism $\bigoplus_{\rho \in \Sigma(1)} \cO_X(D_\rho) \to T_X$ is defined as follows:
the image of a local section $\chi^w$ of $\cO_X(D_\rho)$ over an affine toric subscheme $U \subseteq X$ is the derivation $\partial(\rho,w) \colon \cO_U \to \cO_U$ defined by
\begin{equation} \label{eq:definition_partial_rho_w}
	\partial(\rho,w)(\chi^{u}) = \langle u, \rho \rangle \chi^{u+w}.
\end{equation}

\begin{example} \label{ex:P^n}
For $X = \PP^n$ the diagram in Proposition~\ref{prop:ses} is
\begin{equation*}
	\xymatrix{
		0 \ar[r] & \cO_{\PP^n} \ar[r]^{\! \! \! \! \! \! \! \! \! \! \! \! \!  \! \! \! \!  (x_0, \dots, x_n)} \ar[d]^\beta & \cO_{\PP^n}(1)^{\oplus (n+1)} \ar[d]^{\left( \frac{\partial F}{\partial x_i} \right)_i} \ar[r] & T_{\PP^n} \ar[d]^d \ar[r] & 0 \\
		0 	\ar[r] & \cO_{\PP^n} \ar[r]^{\! \! F} & \cO_{\PP^n}(\beta) \ar[r] & N_{D/\PP^n} \ar[r] & 0
	}
\end{equation*}
where $F \in \KK[x_0, \dots, x_n]$ is a homogeneous polynomial of degree $\beta$.
The commutativity of the left square is the Euler identity
\begin{equation*}
\beta F = \sum_{i=0}^n x_i \frac{\partial F}{\partial x_i}.
\end{equation*}

Let us now analyse the commutativity of the right square. Restrict to the affine chart $U = \{ x_0 \neq 0 \}$ which is isomorphic to $\AA^n$ with affine coordinates $y_k = x_k / x_0$ for $k=1, \dots, n$. Consider the dehomogenisation
\[
f(y_1, \dots, y_n) = F(x_0, \dots, x_n) / x_0^\beta = F(1, y_1, \dots, y_n).
\]
We have the equalities
\begin{equation} \label{eq:P^n_partial_derivative_1}
\frac{\partial F}{\partial x_k}(1, y_1, \dots, y_n) = \frac{\partial f}{\partial y_k} \qquad \text{for } k=1, \dots, n
\end{equation}
and
\begin{equation} \label{eq:P^n_partial_derivative_0}
 \beta f =  \sum_{k=1}^n y_k \frac{\partial f}{\partial y_k} + \frac{\partial F}{\partial x_0}(1, y_1, \dots, y_n).
\end{equation}
By using the trivialisations
\begin{gather*}
\cO_U \simeq \cO_{\PP^n}(1) \vert_U \qquad g \mapsto x_0 g, \\
\cO_U \simeq \cO_{\PP^n}(\beta) \vert_U \qquad g \mapsto x_0^\beta g, \\
\cO_U^{\oplus n} \simeq T_{\PP^n} \vert_U \qquad (g_1, \dots, g_n) \mapsto \sum_{i=1}^n g_i \frac{\partial}{\partial y_i} \quad \text{ and } \quad \left( \partial(y_1), \dots, \partial(y_n) \right) \mapsfrom \partial,
\end{gather*}
the right square in the diagram becomes
\begin{equation*}
	\xymatrix{
	\Gamma(U,\cO_U)^{\oplus (n+1)} \ar[rrr]^{\begin{pmatrix}
			-y_1 & 1 \\
			\vdots & & \ddots \\
			-y_n & & & 1
	\end{pmatrix}} 
\ar[dd]_{
	\left( \frac{\partial F}{\partial x_i}(1,y_1, \dots, y_n)  \right)_{i=0, \dots, n}
}
& & &
\Gamma(U,\cO_U)^{\oplus n}
\ar[dd]^{
\left( \overline{\frac{\partial f}{\partial y_i}} \right)_{i=1,\dots, n}
} \\
\\
\Gamma(U,\cO_U) \ar[rrr]^{\overline{\cdot}}
& & &
\Gamma(U,\cO_{D \cap U})
}
\end{equation*}
where $\overline{\cdot} \colon \Gamma(U,\cO_U) = \KK[y_1, \dots, y_n] \onto \Gamma(U, \cO_{D \cap U}) = \KK[y_1, \dots, y_n] / (f)$ denotes the projection modulo $f$.
This square commutes because of \eqref{eq:P^n_partial_derivative_1} and \eqref{eq:P^n_partial_derivative_0}.
\end{example}

\begin{proof}[Proof of Proposition~\ref{prop:ses}]
The existence of the two short exact sequences is clear.
The commutativity of the left square is the Euler relation \cite[Exercise~8.1.8]{cls}.
It remains to prove the commutativity of the right square. We proceed in a way analogous to Example~\ref{ex:P^n}.

We restrict to the affine toric subscheme $U \subseteq X$ associated to an $n$-dimensional cone $\sigma \in \Sigma$.
Assume that the rays of $\sigma$ are $\rho_1, \dots, \rho_n$, and $\rho_{n+1}, \dots, \rho_{n+r}$ are the rays of $\Sigma$ not in $\sigma$.
Here $r$ is the Picard rank of $X$.
For brevity, set $x_i := x_{\rho_i}$, $D_i := D_{\rho_i}$, $\beta_i := \beta_{\rho_i}$.

Since $X$ is smooth, $\rho_1, \dots, \rho_n$ form a basis of the lattice $N$. Hence we can write
$\rho_j = - \sum_{k=1}^n a_{kj} \rho_k$ for every $j = n+1, \dots, n+r$.
We have that $\beta_{n+1}, \dots, \beta_{n+r}$ is a basis of $\Pic(X)$ and $\beta_k = \sum_{j=n+1} a_{kj} \beta_j$ for every $k=1, \dots, n$.
We also write $\beta = \sum_{j=n+1}^{n+r} b_j \beta_j$.
Let $w_1, \dots, w_n \in M$ be the dual basis of $\rho_1, \dots, \rho_n$.

Set $y_k := \chi^{w_k}$ for every $k=1, \dots, n$. It is clear that $U$ is isomorphic to $\AA^n$ with coordinates $y_1, \dots, y_n$.
Under the Cox isomorphism \cite[Equation~(5.3.1)]{cls} $y_k = x_k \prod_{j=n+1}^{n+r} x_j^{- a_{kj}}$.
The dehomogenisation of $F$ is
\begin{equation*}
	f(y_1, \dots, y_n) = \frac{F}{\prod_{j=n+1}^{n+r} x_j^{b_j}} = F(y_1, \dots, y_n, 1, \dots, 1).
\end{equation*}
There are the following Euler relations:
\begin{equation*}
	b_j F =  \sum_{k=1}^n a_{kj} x_k \frac{\partial F}{\partial x_k} + x_j \frac{\partial F}{\partial x_j} \qquad \text{for  } j=n+1, \dots, n+r.
\end{equation*}
We have
\begin{equation} \label{eq:toric_partial_derivative_y_k}
\frac{\partial f}{\partial y_k} = \frac{\partial F}{\partial x_k}(y_1, \dots, y_n, 1, \dots, 1) \qquad \text{for } k = 1, \dots, n
\end{equation}
and
\begin{equation} \label{eq:toric_partial_derivative_x_j}
	b_j f = \sum_{k=1}^n a_{kj} y_k \frac{\partial f}{\partial y_k} + \frac{\partial F}{\partial x_j}(y_1, \dots, y_n, 1, \dots, 1) \qquad \text{for } j = n+1, \dots, n+r.
\end{equation}

We want to show that the diagram
\begin{equation*}
	\xymatrix{
	\bigoplus_{i=1}^{n+r} \Gamma(U, \cO_X(D_i)) \ar[d]_{\frac{\partial F}{\partial x_i}} \ar[r] & \Gamma(U, T_X) \ar[d]^d \\
	\Gamma(U, \cO_X(D)) \ar[r] & \Gamma(U, N_{D/X})
}
\end{equation*}
commutes.
By using the trivialisations
\begin{gather*}
	\Gamma(U, T_X) \simeq \Gamma (U, \cO_U)^{\oplus n} \qquad \partial \mapsto \left( \partial(y_1), \dots, \partial(y_n) \right) \\
	\Gamma(U, \cO_U) \simeq \Gamma(U, \cO_X(D_k)) \qquad g \mapsto g y_k^{-1} \qquad \text{for } k=1, \dots, n \\
	\Gamma(U, \cO_U) = \Gamma(U, \cO_X(D_j)) \qquad \text{for } j=n+1, \dots, n+r
\end{gather*}
and the two following consequences of \eqref{eq:definition_partial_rho_w}
\begin{gather*}
\partial(\rho_i, -w_i)(\chi^{w_k}) = \langle w_k, \rho_i \rangle \chi^{w_k - w_i} = \delta_{ik} \qquad i=1, \dots, n, \\
\partial(\rho_j,0)(\chi^{w_k}) = \langle w_k, \rho_j \rangle \chi^{w_k} = - a_{kj} y_k  \qquad j=n+1, \dots, n+r,
\end{gather*}
this diagram becomes
\begin{equation*}
	\xymatrix{
\Gamma(U, \cO_U)^{\oplus (n+r)}
\ar[rrrr]^{\ \
\left(\begin{array}{ccc|ccc}
1 & & \\
& \ddots & & & - a_{kj} y_k & \\
& & 1
\end{array}\right)
}
\ar[dd]^{
	\left( \frac{\partial F}{\partial x_i}(y_1, \dots, y_n,1,\dots,1)  \right)_{i=1, \dots, n+r}}
& & & &
\Gamma(U, \cO_U)^{\oplus n}
\ar[dd]^{
	\left( \overline{\frac{\partial f}{\partial y_i}} \right)_{i=1,\dots, n}
}
\\
\\
\Gamma(U, \cO_U)
\ar[rrrr]^{\overline{\cdot}}
& & & &
\Gamma(U, \cO_{D \cap U})
}
\end{equation*}
 where $\overline{\cdot} \colon \Gamma(U, \cO_U) = \KK[y_1, \dots, y_n] \onto \Gamma(U, \cO_{D \cap U}) = \KK[y_1, \dots, y_n] / (f)$ denotes the projection modulo $f$.
 This square commutes because of 
 \eqref{eq:toric_partial_derivative_y_k} and \eqref{eq:toric_partial_derivative_x_j}.
\end{proof}

The restriction to every affine toric open subscheme $U \subseteq X$ of the sheaf homomorphism $\rN_1(X) \otimes_\ZZ \cO_X \to \oplus_{\rho \in \Sigma(1)} \cO_X(D_\rho)$ appearing in the top short exact sequence in Proposition~\ref{prop:ses} is homogeneous with respect to the $M$-grading.
This implies that the cohomology groups $\rH^i(X,T_X)$ inherit an $M$-grading.
We denote by $\rH^i(X,T_X)_u$ the part of $\rH^i(X,T_X)$ of degree $u \in M$.

\begin{remark}
One can prove that the cup product \eqref{eq:cup_product} preserves the $M$-grading;
in \cite{ilten_turo} Ilten and Turo give a combinatorial description of its homogeneous parts with respect to the $M$-grading.
\end{remark}

\section{The example}\label{sec:example}

We construct an explicit smooth projective toric variety of each dimension $n \geq 3$ that turns out to give a counterexample to Question~\ref{question}.

\begin{prop} \label{prop:obstructed}
	For every $n \geq 3$, there exist a smooth projective toric $n$-fold $X$  and elements $u', u'' \in M$, where $M$ denotes the character lattice of the big torus of $X$, such that:
	\begin{enumerate}
		\item the cup product $\rH^1(X,T_X)_{u'} \otimes_\KK \rH^1(X,T_X)_{u''} \to \rH^2(X,T_X)_{u'+u''}$ is non-zero,
		\item $\rH^1(X,\cO_X(D))_{u'} = 0$ and $\rH^1(X,\cO_X(D))_{u''} = 0$, where $D$ is the toric boundary of $X$.
	\end{enumerate}
\end{prop}

The rest of this section is devoted to the proof of Proposition~\ref{prop:obstructed}.
Firstly, we consider the case $n=3$; the $3$-fold we construct is the projectivisation of a split rank-$2$ vector bundle on the second Hirzebruch surface.
Secondly, in dimension $n \geq 4$ we consider the product of this $3$-fold with $\PP^{n-3}$.

\begin{example} \label{ex:3fold}
	Consider the lattice $N = \ZZ^3$ and its dual $M$.
	Let $\rho_1, \dots, \rho_6 \in N$ be the columns of the matrix
	\[
	\begin{pmatrix}
		1&0&-1&0&0&0\\
		0&1&2&-1&0&0\\
		0&0&-2&3&1&-1
	\end{pmatrix}.
	\]
	Let $\Sigma$ be the fan in $N$ with rays given by $\rho_1, \dots, \rho_6$ and with the following $3$-dimensional cones: $\sigma_{125}$, $\sigma_{126}$, $\sigma_{145}$, $\sigma_{146}$, $\sigma_{235}$, $\sigma_{236}$, $\sigma_{345}$, $\sigma_{346}$, where $\sigma_{ijk}$ denotes the cone with rays $\rho_i, \rho_j, \rho_k$.
	The fan $\Sigma$ can be visualised by looking at Figure~\ref{fig:cones}: by considering $\rho_6$ as a vertex at infinity, we can describe the maximal cones of $\Sigma$ by the $2$-simplices of the simplicial complex in Figure~\ref{fig:cones}.
	
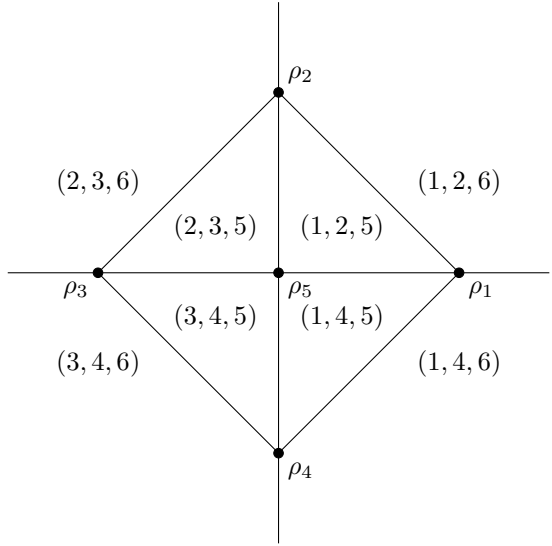
\begin{figure}
\centering
\input{Figure_cones}
\caption{A schematic picture of the fan $\Sigma$ in Example~\ref{ex:3fold}.
	The four triangles correspond to the four $3$-dimensional cones in $\Sigma$ which contain $\rho_5$,
	and the four $2$-dimensional unbounded regions correspond to the four $3$-dimensional cones in $\Sigma$ which contain $\rho_6$.}
\label{fig:cones}
\end{figure}

Let $X$ be the toric variety associated to the fan $\Sigma$; it is a smooth projective $3$-fold with Picard rank $3$.
One can prove that $X$ is the projectivisation of a split rank-$2$ vector bundle over the second Hirzebruch surface -- see \cite[Theorem~3.6]{robins}.

For a torus invariant divisor $D=\sum_{\rho\in \Sigma(1)} a_{\rho}D_{\rho}$ and $u\in M$, we consider the simplicial complex
\begin{equation*}
V_{D,u}= \bigcup _{\sigma \in \Sigma} \text{conv} \left( \rho \in \sigma(1) \mid \langle u, \rho \rangle < -a_{\rho} \right).
\end{equation*}
When $D=D_{\rho}$, we simply denote it by $V_{\rho,u}$.
Various $V_{D,u}$'s are depicted in Figure~\ref{fig:my_label}.

\begin{figure}
	\centering
	\input{Figure_VDu}
	\caption{The simplicial complex $V_{D,u}$ in red for different values of $D$ and $u$, in Example~\ref{ex:3fold}}
	\label{fig:my_label}
\end{figure}
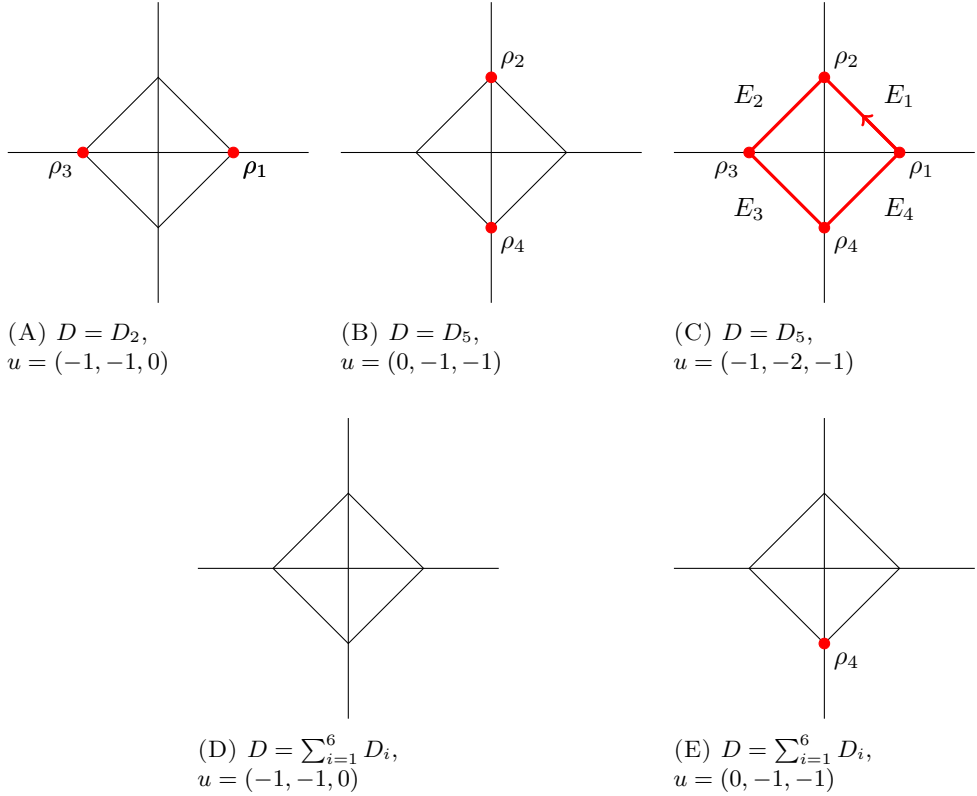

We have the following isomorphisms from \cite[Theorem~9.1.3]{cls} and \cite[Proposition~3.1]{ilten_turo}:
\begin{gather}
	    \rH^i(X,\mathcal{O}_{X}(D))_u \simeq  \widetilde{\rH}^{i-1}(V_{D,u},\KK), \label{eq:OXD} \\
	\rH^i(X,T_X)_u \simeq \bigoplus_{\rho \in \Sigma(1)} \widetilde{\rH}^{i-1}(V_{\rho,u},\KK). \label{eq:TX}
\end{gather}

Now fix $u'=(-1,-1,0)$ and $u''=(0,-1,-1)$ in $M$.
By looking at (A), (B), (C) in Figure~\ref{fig:my_label} and by using \eqref{eq:TX}, we get 
\begin{equation*}
\rH^1(X,T_X)_{u'}=\KK, \quad
\rH^1(X,T_X)_{u''}=\KK, \quad \rH^2(X,T_X)_{u'+u''}=\KK.
\end{equation*}
Now we have all the necessary conditions in \cite[Theorem~4.3]{ilten_turo}, but it is not immediately clear whether the cup product is non-zero.
To see this, we proceed as follows.

We use the notation and the constructions from \cite[\S5]{ilten_turo}.
Consider $Z=\rho_1$, $Z'=\rho_2$ and the simple $1$-cycle $\alpha= E_1+E_2+E_3+E_4$ in $V_{D_5, u'+u''}$ as in  Figure~\ref{fig:sub4}.
We have then $\alpha(Z)=\{E_1,E_4\}$, $\alpha(Z')=\{E_1,E_2\}$, $b_1=b_4=1$, $b_2=b_3=0$.
It is immediate to show $Z *_\alpha Z' \neq 0$.
Now  \cite[Theorem~5.3]{ilten_turo} implies that the cup product is non-zero.
This proves (1) in Proposition~\ref{prop:obstructed}.

Now we consider the toric boundary $D = \sum_{i=1}^6 D_i$ of $X$.
From Figure~\ref{fig:sub5} and Figure~\ref{fig:sub6} we can see that $V_{D,u'}$ is empty and $V_{D,u''}$ is a point.
Therefore by \eqref{eq:OXD} we get $\rH^1(X,\mathcal{O}_X(D))_{u'}=0$ and $\rH^1(X,\mathcal{O}_X(D))_{u''}=0$.
This proves (2) in Proposition~\ref{prop:obstructed}.
\end{example}

\begin{remark}
Let $X$ be the smooth projective toric $3$-fold in Example~\ref{ex:3fold}.
One can prove that $\rH^1(X, T_X)$ has dimension $3$ and the following degrees: $u'=(-1,-1,0)$, $u''=(0,-1,-1)$, $(-1,0,1)$.
Moreover, $\rH^2(X,T_X)$ has dimension $1$ and degree $u' + u'' = (-1,-2,-1)$.
Using  the $M$-grading and the fact that the cup product is non-zero on $\rH^1(X,T_X)_{u'} \otimes_\KK \rH^1(X,T_X)_{u''}$,  one can show that the hull of $\Def{X}$ is $\KK \pow{t_1, t_2, t_3} / (t_1 t_2)$.
\end{remark}

\begin{remark}
One can check that the toric variety in \cite[\S6]{ilten_turo} is a smooth projective toric $3$-fold which has Picard rank $6$ and satisfies the conditions of Proposition~\ref{prop:obstructed}.
The variety described in Example~\ref{ex:3fold} is much simpler and has minimal Picard rank among obstructed smooth projective toric varieties.
\end{remark}

\begin{proof}[{Proof of Proposition~\ref{prop:obstructed}}]
Let $X$ be the smooth projective toric $3$-fold in Example~\ref{ex:3fold}. 
Set $\widetilde{X}=X\times \PP^{n-3}$; the character lattice of the big torus of $\widetilde{X}$ is $\widetilde{M} = M \oplus \ZZ^{n-3}$.
We will show that $\widetilde{X}$ satisfies the conditions (1) and (2) for $v'=(u',0) \in \widetilde{M}$ and $v''=(u'',0) \in \widetilde{M}$.

 If $\Sigma$ is the fan of $X$ and $\Sigma_0$ is the fan of $\mathbb{P}^{n-3}$, then $\widetilde{X}$ can be described by the fan $\widetilde{\Sigma}=\Sigma \times \Sigma_0$.
 The image of $\rho_i\in \Sigma(1)$ in $\widetilde{\Sigma}(1)$ is denoted by $\widetilde{\rho_i}$, and the image of $\tau_i\in \Sigma_0(1)$ in $\widetilde{\Sigma}(1)$ is denoted by $\widetilde{\tau_i}$.
 
The first-order deformation of $X$ corresponding to the degree $u'$ (resp.\ $u''$) induces a first-order deformation of $\widetilde{X} = X\times \PP^{n-3}$ corresponding to the degree $v'$ (resp.\ $v''$) by deforming only the first factor.
From a combinatorial point of view this can be seen from \eqref{eq:TX} and  by observing that
 $V_{{\rho_2},u'}=V_{{\widetilde{\rho_2}},v'}$ and $V_{{\rho_5},u''}=V_{{\widetilde{\rho_5}},v''}$.
 
 For computing the cup product, first observe that $\langle \widetilde{\tau_i},v'\rangle=0= \langle \widetilde{\tau_i},v''\rangle$ for $\tau_i \in \Sigma_0(1)$. Hence by restricting the attention to the cones generated by the rays $\widetilde{\rho_i}$, following the recipe in \cite[\S5]{ilten_turo} is exactly the same as in Example~\ref{ex:3fold}. Thus, we can see that the cup product is non-zero. 
 
 Finally, observe that $V_{D,u'}=V_{\widetilde{D},v'}$ and $V_{D,u''}=V_{\widetilde{D},v''}$, where $D$ (resp.\ $\widetilde{D}$) is the toric boundary of $X$ (resp.\ $\widetilde{X}$). Hence the condition (2) is immediate from \eqref{eq:OXD}. 
\end{proof}

\section{Proof of Theorem~\ref{thm:main}} \label{sec:proof}

Let $X$ be one of the smooth projective toric varieties constructed in Proposition~\ref{prop:obstructed}, let $D$ be the toric boundary of $X$, and let $u', u'' \in M$ be the degrees satisfying the conditions (1) and (2) in Proposition~\ref{prop:obstructed}.

Let $\xi' \in \rH^1(T_X)_{u'}$ and $\xi'' \in \rH^1(T_X)_{u''}$ be such that the cup product $b(\xi', \xi'') \in \rH^2(T_X)_{u' + u''}$ is non-zero -- they exist by (1).
By the polarisation identity for quadratic forms we have that either $q(\xi' + \xi'') \neq 0$ or $q(\xi' - \xi'') \neq 0$.
Hence there exists a linear combination $\xi$ of $\xi'$ and $\xi''$ such that $q(\xi)\neq 0$.
This implies that the first-order deformation of $X$ associated to $\xi$ cannot be extended to the second order.

Let $\Sigma$ be the fan defining the toric variety $X$.
Consider the monomial $F = \prod_{\rho \in \Sigma(1)} x_\rho$ in the Cox ring of $X$.
The zero-locus of $F$ is exactly the toric boundary $D$ of $X$.
Since $F$ is a monomial, the homomorphism $\Gamma(U, \cO_X) \to \Gamma(U, \cO_X(D))$ given by the multiplication by $F$ preserves the natural $M$-gradings for every open affine toric subscheme $U \subseteq X$.
By the bottom exact sequence in the diagram in Proposition~\ref{prop:ses}, the cohomology groups $\rH^i(N_{D/X})$ carry natural $M$-gradings.
Since $X$ is rational, $\cO_X$ does not have higher cohomology, therefore we have an isomorphism of $M$-graded vector spaces
\begin{equation*}
	\rH^1(\cO_X(D)) \simeq \rH^1(N_{D/X}).
\end{equation*}
By (2) in Proposition~\ref{prop:obstructed} we have $\rH^1(N_{D/X})_{u'} = 0$ and $\rH^1(N_{D/X})_{u''} = 0$.

Since the derivatives of $F$ are monomials in Cox coordinates, the multiplication by $\frac{\partial F}{\partial x_\rho}$ preserves the $M$-grading when restricted to every open affine toric subscheme of $X$. The commutativity of the diagram in Proposition~\ref{prop:ses} implies that, for every $i$, the homomorphism
\begin{equation*}
	\rH^i(T_X) \longrightarrow \rH^i(N_{D/X})
\end{equation*}
induced by \eqref{eq:d} preserves the $M$-grading.
From the above result on the $M$-grading of $\rH^1(N_{D/X})$, we deduce that the two elements $\xi', \xi'' \in \rH^1(T_X)$ maps to zero in $\rH^1(N_{D/X})$.
Therefore $\xi$ maps to zero in $\rH^1(N_{D/X})$.

By the long exact sequence in \S\ref{sec:defo_pairs} we have that $\xi$ lies in the image of $\HH^1(A^\bullet) \to \rH^1(T_X)$.
This implies that the first-order deformation $\xi$ of $X$ can be lifted to a first-order deformation $\eta$ of the pair $(X,D)$.
Since $\xi$ cannot be extended to the second order, also the first-order deformation $\eta$ of $(X,D)$ cannot be extended to the second order.
In particular, $\Def{(X,D)}$ is not smooth.

\bibliography{Biblio_logCY}

\end{document}

%% file: Commands.tex
\renewcommand{\setminus}{\smallsetminus}

\newcommand{\into}{\hookrightarrow} 
\newcommand{\onto}{\twoheadrightarrow} 

\DeclareMathOperator{\Hom}{Hom} 
\newcommand{\cHom}{\mathcal{H}om}
\DeclareMathOperator{\Spec}{Spec} 
\DeclareMathOperator{\Pic}{Pic}

\def\pow#1{ \llbracket  #1 \rrbracket }

\newcommand{\Def}[1]{\mathrm{Def}_{#1}}

\newcommand\cO{\mathcal{O}}

\renewcommand\AA{\mathbb{A}}
\newcommand\CC{\mathbb{C}}
\newcommand\FF{\mathbb{F}}
\newcommand\HH{\mathbb{H}}
\newcommand{\KK}{\Bbbk}

\newcommand\PP{\mathbb{P}}

\newcommand\ZZ{\mathbb{Z}}

\newcommand\rH{\mathrm{H}}
\newcommand\rN{\mathrm{N}}



%% file: Figure_cones.tex
    \begin{tikzpicture}[scale=0.6]
    \draw (4,0)-- (0,4)--(-4,0)--(0,-4)--cycle;
    \draw (0,0)--(6,0);
    \draw (0,0)--(0,6);
    \draw (0,0)--(-6,0);
    \draw (0,0)--(0,-6);

    \node at (1.4,1){$(1,2,5)$};
    \node at (-1.4,1){$(2,3,5)$};
    \node at (-1.4,-1){$(3,4,5)$};
    \node at (1.4,-1){$(1,4,5)$};
    \node at (4,2){$(1,2,6)$};
    \node at (-4,2){$(2,3,6)$};
    \node at (-4,-2){$(3,4,6)$};
    \node at (4,-2){$(1,4,6)$};
    
     \filldraw[black] (0,0) circle (3pt)node[anchor=north west] {$\rho_5$};
     \filldraw[black] (4,0) circle (3pt)node[anchor=north west] {$\rho_1$};
     \filldraw[black] (0,4) circle (3pt)node[anchor=south west] {$\rho_2$};
     \filldraw[black] (-4,0) circle (3pt)node[anchor=north east] {$\rho_3$};
     \filldraw[black] (0,-4) circle (3pt)node[anchor=north west] {$\rho_4$};
\end{tikzpicture}

%% file: Figure_VDu.tex
    \begin{subfigure}{.3\textwidth}
    \centering
    \begin{tikzpicture}[scale=0.25]
    \draw (4,0)--(0,4)--(-4,0)--(0,-4)--cycle;
    \draw (0,0)--(8,0);
    \draw (0,0)--(0,8);
    \draw (0,0)--(-8,0);
    \draw (0,0)--(0,-8);
    
    \node[anchor=north west] at (4,0){$\rho_1$};
    
    \filldraw[red] (4,0) circle (8pt);
    \node[anchor=north west] at (4,0){$\rho_1$};
    \filldraw[red] (-4,0) circle (8pt); \node[anchor=north east] at (-4,0) {$\rho_3$};
    \end{tikzpicture}
    
    \caption{$D=D_2,\\u=(-1,-1,0)$}
    \label{fig:sub1}
    \end{subfigure}\hfill
    \begin{subfigure}{.3\textwidth}
    \begin{tikzpicture}[scale=0.25]
    \draw (4,0)-- (0,4)--(-4,0)--(0,-4)--cycle;
    \draw (0,0)--(8,0);
    \draw (0,0)--(0,8);
    \draw (0,0)--(-8,0);
    \draw (0,0)--(0,-8);
    \filldraw[red] (0,4) circle (8pt);
    \node[anchor=south west] at (0,4) {$\rho_2$};
    \filldraw[red] (0,-4) circle (8pt);
    \node[anchor=north west] at (0,-4) {$\rho_4$};
    \end{tikzpicture}
    \caption{$D=D_5,\\u=(0,-1,-1)$}
    \label{fig:sub3}
    \end{subfigure}\hfill
    \begin{subfigure}{.3\textwidth}
    \begin{tikzpicture}[scale=0.25]
    \draw (4,0)-- (0,4)--(-4,0)--(0,-4)--cycle;
   
    \draw (0,0)--(8,0);
    \draw (0,0)--(0,8);
    \draw (0,0)--(-8,0);
    \draw (0,0)--(0,-8);
    \filldraw[red] (4,0) circle (8pt);
    \node[anchor=north west] at (4,0){$\rho_1$};
    \filldraw[red] (-4,0) circle (8pt);
    \node[anchor=south west] at (0,4) {$\rho_2$};
    \filldraw[red] (0,4) circle (8pt);
    \node[anchor=north east] at (-4,0) {$\rho_3$};
    \filldraw[red] (0,-4) circle (8pt);
    \node[anchor=north west] at (0,-4) {$\rho_4$};
    
    \draw[very thick, red] (4,0)--(0,4)--(-4,0)--(0,-4)--cycle; \draw[->,very thick, red](4,0)--(2,2);
    
    \node at (4,3){$E_1$};
    \node at (-4,3){$E_2$};
    \node at (-4,-3){$E_3$};
    \node at (4,-3){$E_4$};
    \end{tikzpicture}
    \caption{$D=D_5,\\u=(-1,-2,-1)$}
    \label{fig:sub4}
    \end{subfigure}

    \vspace{5mm}   
    
        \hfill
    \begin{subfigure}{.3\textwidth}
    \centering
    \begin{tikzpicture}[scale=0.25]
    \draw (4,0)--(0,4)--(-4,0)--(0,-4)--cycle;
    \draw (0,0)--(8,0);
    \draw (0,0)--(0,8);
    \draw (0,0)--(-8,0);
    \draw (0,0)--(0,-8);
    \end{tikzpicture}
    \caption{$D=\sum_{i=1}^6 D_i, \\ u=(-1,-1,0)$}
    \label{fig:sub5}
    \end{subfigure}\hfill
    \begin{subfigure}{.3\textwidth}
    \begin{tikzpicture}[scale=0.25]
    \draw (4,0)-- (0,4)--(-4,0)--(0,-4)--cycle;
    \draw (0,0)--(8,0);
    \draw (0,0)--(0,8);
    \draw (0,0)--(-8,0);
    \draw (0,0)--(0,-8);
    \filldraw[red] (0,-4) circle (8pt);
    \node[anchor=north west] at (0,-4) {$\rho_4$};
    \end{tikzpicture}
    \caption{$D=\sum_{i=1}^6 D_i,\\u=(0,-1,-1)$}
    \label{fig:sub6}
    \end{subfigure}